\documentclass{amsart}
\usepackage{graphicx}
\usepackage{amssymb,amscd,amsthm,amsxtra}
\usepackage{latexsym}
\usepackage{epsfig}

\newtheorem{thm}{Theorem}[section]

\newtheorem{lem}[thm]{Lemma}

\theoremstyle{definition}
\newtheorem{defn}[thm]{Definition}
\theoremstyle{remark}

\numberwithin{equation}{section}

\newcommand{\R}{\mathbb R}

\newcommand{\eps}{\varepsilon}

\newcommand{\p}{\partial}

\newcommand{\comment}[1]{}

\begin{document}

\title[A note on interior estimates]{A note on interior $W^{2,1+ \eps}$ estimates for the Monge-Amp\`ere equation}
\author{G. De Philippis}
\address{Scuola Normale Superiore di Pisa, 56126 Pisa, Italy}
\email{\tt  guido.dephilippis@sns.it}

\author{A. Figalli}
\address{Department of Mathematics, The University of Texas at Austin, Austin, TX 78712 USA}
\email{\tt  figalli@math.utexas.edu}

\author{O. Savin}
\address{Department of Mathematics, Columbia University, New York, NY 10027 USA}
\email{\tt  savin@math.columbia.edu}

\thanks{A. Figalli was partially supported by NSF grant 0969962.
O. Savin was partially supported by NSF grant 0701037.}

\begin{abstract}
By a variant of the techniques introduced by the first two authors in \cite{DF}
to prove that second derivatives of solutions to the Monge-Amp\`ere equation are locally in $L\log L$,
we obtain interior $W^{2,1+\eps}$ estimates.

\end{abstract}
\maketitle

\section{Introduction}

Interior $W^{2,p}$ estimates for solutions to the Monge-Amp\`ere equation with bounded right hand side  
\begin{equation}\label{eq}
\det D^2 u = f \quad \mbox{in $\Omega$}, \quad u=0 \quad \mbox{on $\p \Omega$,} \quad 0< \lambda \le f \le \Lambda,
\end{equation} 
were obtained by Caffarelli in \cite{C} under the assumption that $|f-1|\leq \eps(p)$ locally. In particular $u \in W^{2,p}_{\rm loc}$ for any $p < \infty$ if $f$ is continuous. 

Whenever $f$ has large oscillation, $W^{2,p}$ estimates are not expected to hold for large values of $p$. Indeed
Wang showed in \cite{W} that for any $p>1$ there are homogenous solutions to \eqref{eq} of the type
$$u(tx, t^\alpha y) =t^{1+\alpha} u(x, y) \quad \mbox{for $t > 0$,}$$
 which are not in $W^{2,p}$.

Recently the first two authors, motivated by a problem arising from the semigeostrophic equation \cite{ACDF,ACDF2},
 showed that interior $W^{2,1}$ estimates hold for the equation \eqref{eq} \cite{DF}. In fact they proved higher integrability in the sense that 
 $$\|D^2 u\| \left |\log \|D^2 u\|\right |^k \in L^1_{\rm loc} \quad \quad \forall\, k \ge 0.$$ 

In this short note we obtain interior $W^{2,1+\eps}$ estimates for some small
$\eps=\eps(n,\lambda,\Lambda)>0$. In view of the examples in \cite{W} this result is optimal.
We use the same ideas as in \cite{DF}, which mainly consist in looking to the $L^1$ norm of $\|D^2 u\|$
over the sections of $u$ itself and prove some decay estimates. Below we give the precise statement.

\begin{thm}\label{thm}
Let $u: \overline \Omega \to \R$,
\begin{equation*}
u=0 \quad \mbox{on $\p \Omega$}, \quad \quad B_1\subset \Omega \subset B_n,
\end{equation*}
be a continuous convex solution to the Monge-Amp\`ere equation
\begin{equation}\label{MA}
\det D^2 u=f(x) \quad \mbox{in $\Omega$}, \quad \quad 0< \lambda \le f \le \Lambda,
\end{equation}
for some positive constants $\lambda$, $\Lambda$. Then 
$$\|u\|_{W^{2,1+\eps}(\Omega')} \le C, \quad \quad \mbox {with} \quad \Omega':=\{ u< -\|u\|_{L^\infty}/2\} ,$$
where $\eps,C>0$ are universal constants depending on $n$, $\lambda$, and $\Lambda$ only.
\end{thm}
 By a standard covering argument (see for instance \cite[Proof of (3.1)]{DF}), this implies that $u \in W^{2,1+\eps}_{\rm loc}(\Omega)$.

Theorem \ref{thm} follows by slightly modifying the strategy in \cite{DF}: We use a 
covering lemma that is better localized (see Lemma \ref{basic}) to obtain a 
geometric decay of the ``truncated'' $L^1$ energy for $\|D^2 u\|$ (see Lemma \ref{log}).

 We also give a second proof of Theorem \ref{thm} based on the following observation:
In view of \cite{DF} the $L^1$ norm of $\|D^2 u\|$ decays on sets of small measure:
 $$|\{ \|D^2 u \| \ge M\}| \le \frac{C}{M\log M},$$
 for an appropriate universal constant $C>0$ and for any $M$ large.
 In particular, choosing first $M$ sufficiently large and then taking $\eps>0$ small enough,
we deduce (a localized version of) the bound
 $$
 |\{ \|D^2 u \| \ge M\}| \le \frac{1}{M^{1+\eps}} |\{ \|D^2 u \| \ge 1\}|
 $$
Applying this estimate at all scales (together with a covering lemma) leads to
the local $W^{2,1+ \eps}$ integrability for $\|D^2 u\|$.

We believe that both approaches are of interest, and for this reason we include both.
In particular, the first approach gives a direct proof of the $W^{2,1+\eps}_{\rm loc}$
regularity without passing through the $L\log L$ estimate.

We remark that the estimate of Theorem \ref{thm} holds under slightly weaker assumptions on the right hand side.
 Precisely if $$\det D^2 u =\mu$$
with $\mu$ being a finite combination of measures which are bounded between two multiples of a nonnegative polynomial,
then the $W^{2,1+\eps}_{\rm loc}$
regularity still holds (see Theorem \ref{thm2} for a precise statement).

The paper is organized as follows. In section \ref{sec:not} we introduce the notation and some basic properties of solution
to the Monge-Amp\`ere equation with bounded right hand side.
Then, in section \ref{sec:pf} we show both proofs of Theorem \ref{thm}, together with the extension to polynomial right hand sides.\\

After the writing of this paper was completed, we learned that Schmidt \cite{S} had just obtained the same result
with related but somehow different techniques.

\section{Notation and Preliminaries}\label{sec:not}

{\bf Notation.}
Given a convex function $u:\Omega \to \R$ with $\Omega \subset \R^n$ bounded and convex, we define its section $S_h(x_0)$ centered at $x_0$ at height $h$ as  $$S_h(x_0)=\{x \in \Omega \,: \quad u(x)<u(x_0)+\nabla u(x_0) \cdot (x-x_0)+h\}.$$ 
We also denote by $\overline{S_h}(x_0)$ the closure of $S_h(x_0)$.

The norm $\|A \|$ of an $n \times n$ matrix $A$ is defined as 
$$\|  A \|:=\sup_{|x| \le 1} Ax. $$

We denote by $|F|$ the Lebesgue measure of a measurable set $F$.

Positive constants depending on $n$, $\lambda$, $\Lambda$ are called {\it universal constants}.
In general we denote them by $c$, $C$, $c_i$, $C_i$.

\

Next we state some basic properties of solutions to \eqref{MA}.

\subsection{Scaling properties} 

If $S_h(x_0) \subset \subset \Omega,$
then (see for example \cite{C}) there exists a linear transformation $A:\R^n\to\R^n$, with $\det A=1$, such that
\begin{equation}\label{scal}
\sigma  B_{\sqrt h} \subset A(S_h(x_0)-x_0) \subset \sigma^{-1}  B_{\sqrt h},
\end{equation}
for some $\sigma>0$, small universal. 

\begin{defn}
We say that $S_h(x_0)$ has {\it normalized size} $\alpha$ if $$\alpha:=\|A\|^2$$ for some matrix
$A$ that satisfies the properties above. (Notice that, although $A$ may not be unique,
this definition fixes the value of $\alpha$ up to multiplicative universal constants.)
\end{defn}

It is not difficult to check that if $u$ is $C^2$ in a neighborhood of $x_0$,
then $S_h(x_0)$ has normalized size $\|D^2 u(x_0)\|$ for all small $h>0$ (if necessary we need to lower the value of $\sigma$).

Given a transformation $A$ as in \eqref{scal}, we define $\tilde u$ to be the  rescaling of $u$
\begin{equation}\label{tilde}
\tilde u(\tilde x)=h^{-1}u(x), \quad \quad \tilde x=Tx:=h^{-1/2}A \, (x-x_0).
\end{equation}
Then $\tilde u$ solves an equation in the same class 
$$\det D^2 \tilde u =\tilde f, \quad \quad \mbox{with} \quad \tilde f(\tilde x):=f(x), \quad \quad \lambda \le \tilde f \le \Lambda, $$
and the section $\tilde S_1(0)$ of $\tilde u$ at height 1 is normalized i.e $$ \sigma B_1 \subset \tilde S_1(0) \subset \sigma^{-1} B_1, \quad \quad \tilde S_1(0) = T (S_h(x_0)).$$
Also $$D^2u(x)=A^TD^2\tilde u(\tilde x)A,$$ hence
\begin{equation}\label{A1}
\|D^2u(x)\| \le \|A\|^2 \|D^2\tilde u(\tilde x)\|,
\end{equation}
and  
\begin{equation}\label{A2}
\gamma_1 I \le D^2 \tilde u(\tilde x) \le \gamma_2 I \quad \quad \Rightarrow \quad \quad\gamma_1 \|A\|^2  \le \|D^2  u( x)\| \le \gamma_2 \|A\|^2.
\end{equation}

\

\subsection{Properties of sections}\label{eng}

Caffarelli and Gutierrez showed in \cite{CG} that sections $S_h(x)$ which are compactly included in $\Omega$ have engulfing properties similar to the engulfing properties of balls. In particular we can find $\delta>0$ small universal such that:

1) If $h_1 \le h_2$ and $S_{\delta h_1}(x_1) \cap S_{\delta h_2}(x_2) \ne \emptyset $ then $$S_{\delta h_1}(x_1) \subset S_{h_2}(x_2).$$

2) If $h_1 \le h_2$ and $x_1 \in \overline{S_{h_2}}(x_2) $ then we can find a point $z$ such that
 $$S_{\delta h_1}(z) \subset S_{h_1}(x_1) \cap S_{h_2}(x_2).$$  
 
 3) If $x_1 \in \overline{S_{h_2}}(x_2) $ then $$ S_{\delta h_2}(x_1) \subset S_{2h_2}(x_2).$$
 Now we also state a covering lemma for sections.
\begin{lem}[Vitali covering]\label{l3}
Let $D$ be a compact set in $\Omega$ and assume that to each $x \in D$ we associate a corresponding section $S_h(x)\subset \subset \Omega$. Then we can find a finite number of these sections $S_{h_i}(x_i)$, $i=1,\dots,m$, such that $$D \subset \bigcup_{i=1}^m S_{ h_i}(x_i), \quad \quad \mbox{with $S_{\delta h_i}(x_i) $ disjoint.}$$
\end{lem}

The proof follows as in the standard case:
we first select by comptactness a finite number of sections $S_{\delta h_j}(x_j)$ which cover $D$,
and then choose a maximal disjoint set from these sections, selecting at each step
a section which has maximal height among the ones still available (see the proof of
\cite[Chapter 1, \S 3, Lemma 1]{St} for more details).

\section{Proof of  Theorem \ref{thm}}\label{sec:pf}
 We assume throughout that $u$ is a normalized solution in $S_1(0)$ in the sense that
 $$\det D^2 u=f \quad \mbox{in $\Omega$}, \quad \quad \lambda \le f \le \Lambda,$$
 and
 $$S_2(0) \subset \subset \Omega, \quad \quad \sigma B_1 \subset S_1(0) \subset \sigma^{-1}B_1.$$   
 In this section we show that 
 \begin{equation}\label{est}
 \int_{S_1(0)}\|D^2u\|^{1+\eps} dx \le C,
 \end{equation}
 for some universal constants $\eps>0$ small and $C$ large.
Then Theorem \ref{thm} easily follows from this estimate and a covering argument
based on the engulfing properties of sections.
Without loss of generality we may assume that $u \in C^2$, since the general case follows by approximation.

\subsection{A direct proof of Theorem \ref{thm}}
In this section we give a selfcontained proof of Theorem \ref{thm}.
As already mentioned in the introduction, the idea is to get
a geometric decay for $\int_{\{ \|D^2 u \| \ge M\}}\|D^2u\|$.

\begin{lem}\label{basic}
Assume $0\in \overline{S_t}(y) \subset \subset \Omega$ for some $t \ge 1$ and $y\in \Omega$. Then
$$\int_{S_1(0)} \|D^2 u\| dx \le C_0\, \left | \left \{C_0^{-1}I \le D^2u \le C_0 I \right \} \cap S_\delta(0) \cap S_t(y) \right|,$$
for some $C_0$ large universal.
\end{lem}

\begin{proof} By convexity of $u$ we have
\begin{equation*}\label{inteq}
\int_{S_1(0)} \|D^2 u\| dx \le \int_{S_1(0)} \Delta u \, dx = \int_{\p S_1(0)} u_\nu  \le C_1,
\end{equation*}
where the last inequality follows from the interior  Lipschitz estimate of $u$ in $S_2(0)$. 
 
 The second property in Subsection \ref{eng} gives
$$
S_\delta(0) \cap S_t(y) \supset S_{\delta^2}(z) 
$$
for some point $z$, which implies that
$$|  S_\delta(0) \cap S_t(y)| \ge c_1 $$ for some $c_1>0$ universal. The last two inequalities show that the set
$$\left \{\|D^2 u\| \le 2C_1c_1^{-1} \right \} $$
has at least measure $c_1/2$ inside $S_\delta(0) \cap S_h(y)$. 

Finally, the lower bound on $\det D^2 u$ implies that
$$
C_0^{-1}I \le D^2u \le C_0I\qquad \text{inside $\{\|D^2 u\| \le 2C_1c_1^{-1} \}$,}
$$
and the conclusion follows provided that we choose $C_0$ sufficiently large.
\end{proof}

By rescaling we obtain:

\begin{lem}\label{basic_sc}
Assume $S_{2h}(x_0) \subset \subset \Omega$, and $x_0 \in \overline{S_t}(y)$ for some $t \ge h $.
If $$S_h(x_0) \  \mbox{ has normalized size $\alpha$},$$ then
$$\int_{S_h(x_0)} \|D^2 u\| \, dx \le C_0 \alpha\, \left | \left \{C_0^{-1} \alpha \le \|D^2u \| \le C_0 \alpha\right \} \cap S_{\delta h}(x_0) \cap S_t(y) \right|.$$

\end{lem}

\begin{proof}
The lemma follows by applying Lemma \ref{basic} to the rescaling $\tilde u$
defined in Section 2 (see \eqref{tilde}). More precisely, we notice first that by \eqref{A1} we have
$$\|D^2u(x)\| \le \alpha \|D^2 \tilde u(\tilde x)\|, \quad \quad \tilde x =Tx,$$ hence 
$$ |\det T| \, \, \, \int_{S_h(x_0)}\|D^2 u\| \, dx \le \alpha \int_{\tilde S_1(0)}\|D^2 \tilde u\| \,d \tilde x.$$
Also, by \eqref{A2} we obtain
$$\left \{C_0^{-1}I \le D^2 \tilde u \le C_0 I  \right \} \subset T\left ( \left \{C_0^{-1} \alpha \le \|D^2 u\| \le C_0 \alpha \right  \} \right ).$$
which together with 
$$\tilde S_\delta(0) =T(S_{\delta h}), \quad \tilde S_{t/h}(\tilde y)=T(S_t(y)),$$
implies that 
$$ \left | \left \{C_0^{-1} I \le D^2 \tilde u  \le C_0 I \right \} \cap \tilde S_{\delta }(0) \cap \tilde S_{t/h}(\tilde y) \right| $$
is bounded above by
$$ |\det T| \, \left | \left \{C_0^{-1} \alpha \le \|D^2u \| \le C_0 \alpha \right \} \cap S_{\delta h}(x_0) \cap S_t(y) \right|.$$
The conclusion follows now by applying Lemma \ref{basic} to $\tilde u$.
\end{proof}

Next we denote by $D_k$, $k \geq 0$, the closed sets 
\begin{equation}\label{Dk}
D_k:= \left \{ x \in \overline{S_1}(0): \, \|D^2u(x)\| \ge M^k  \right \},
\end{equation} for some large $M$.
As we show now, Lemma \ref{basic_sc} combined with a covering argument gives a geometric decay for $\int_{D_k}\|D^2 u\|$.

\begin{lem} \label{log}
If $M=C_2$, with $C_2$ a large universal constant, then
$$\int_{D_{k+1}} \|D^2u\| \, dx \le (1-\tau)\int_{D_k}\|D^2 u\| \, dx, $$
for some small universal constant $\tau>0$.
\end{lem}

\begin{proof} Let $M\gg C_0$ (to be fixed later), and
for each $x \in D_{k+1}$ consider a section $$S_h(x) \ \mbox{ of normalized size $\alpha=C_0 M^k$,}$$
which is compactly included in $S_2(0)$. This is possible since for $h \to 0$ the 
normalized size of $S_h(x)$ converges to $\|D^2u(x)\|$ (recall that $u \in C^2$) which is greater than $M^{k+1}> \alpha$,
whereas if $h=\delta$ the normalized size is bounded above by a universal constant and therefore by $\alpha$. 

Now we choose a Vitali cover for $D_{k+1}$ with sections $S_{h_i}(x_i)$, $i=1,\ldots,m$.
Then by Lemma \ref{basic_sc}, for each $i$,
$$\int_{S_{h_i}(x_i)} \|D^2 u\| dx \le C_0^2 M^k\, \left | \left \{M^k \le \|D^2u \| \le C_0^2 M^k\right \} \cap S_{\delta h_i}(x_i) \cap S_1(0) \right|.$$
Adding these inequalities and using  
$$D_{k+1}\, \subset \, \bigcup S_{h_i}(x_i), \quad \quad S_{\delta h_i}(x_i) \ \mbox{disjoint,}  $$ we obtain
\begin{align*}
\int_{D_{k+1}}\|D^2 u\| dx  & \le \,  C_0^2M^k \, \left | \left \{M^k  \le \|D^2u \| \le C_0^2 M^k\right \} \cap  S_1(0)\right | \\
& \le C \int_{D_k \setminus D_{k+1}}\|D^2 u\| dx
\end{align*}
provided $M \geq C_0^2$.
Adding $C \int_{D_{k+1}}\|D^2 u\|$ to both sides of the above inequality,
the conclusion follows with $\tau=1/(1+C)$.
\end{proof}

By the above result, the proof of \eqref{inteq} is immediate: indeed,
by Lemma \ref{log} we easily deduce that there exist $C,\eps>0$ universal such that 
$$
\int_{\{ x \in {S_1}(0): \, \|D^2u(x)\| \ge t  \}}\|D^2 u\| \, dx
\leq C t^{-2\eps}\qquad \forall \,t \geq 1.
$$
Multiplying both sides by $t^{-(1-\eps)}$ and integrating over $[1,\infty)$ we obtain
$$
\int_1^\infty  t^{-(1-\eps)} \int_{\{ x \in {S_1}(0): \, \|D^2u(x)\| \ge t   \}}\|D^2 u\| \, dx\,dt
\leq C\int_1^\infty t^{-1-\eps} =\frac{C}{\eps},
$$
and we conclude using Fubini.

\subsection{A proof by iteration of the $L\log L$ estimate}
We now briefly sketch how \eqref{inteq} could also be easily deduced by applying the $L\log L$ estimate 
from \cite{DF} inside every section, and then performing a covering argument.

First, any $K >0$ we introduce the notation
$$F_K:=  \{\|D^2 u \|\ge K\} \cap S_1(0).$$
\begin{lem}\label{cor}
Suppose $u$ satisfies the assumptions of Lemma \ref{basic}.
Then there exist universal constants $C_0$ and $C_1$ such that, for all $K \geq 2$,
$$|F_K| \le \, \, \frac{C_1}{K\log(K)} \,
\left | \left \{C_0^{-1}I \le D^2u \le C_0 I \right \} \cap S_\delta(0) \cap S_t(y) \right|. $$
\end{lem}

Indeed, from the proof of Lemma \ref{basic} the measure of 
the set appearing on the right hand side is bounded below by a small universal constant
$c_1/2$, while by \cite{DF}
$|F_K| \le C/K\log(K)$ for all $K \geq 2$, hence
$$
|F_K| \leq \frac{2C}{c_1 K\log(K)} \left | \left \{C_0^{-1}I \le D^2u \le C_0 I \right \} \cap S_\delta(0) \cap S_t(y) \right|.
$$

Exactly as in the proof of Lemma \ref{basic_sc}, by rescaling we obtain:

\begin{lem} \label{cors}
Suppose $u$ satisfies the assumptions of Lemma \ref{basic_sc}. Then, 
$$|\{\|D^2 u \|\ge \alpha K\} \cap S_h(x_0)| \le \frac{C_1}{K\log(K)}    \left | \left \{C_0^{-1} \alpha \le \|D^2u\| \right \} \cap S_{\delta h}(x_0) \cap S_t(y) \right|,$$
for all $K \ge 2$.
\end{lem}

Finally, as proved in the next lemma, a covering argument shows that the measure of the sets $D_k$ defined in \eqref{Dk}
decays as $M^{-(1+2\eps)k}$, which gives \eqref{inteq}.
\begin{lem}\label{fin} There exist universal constants $M$ large and $\eps>0$ small such that
$$|D_{k+1}| \le M^{-1-2\eps}|D_k| .$$
\end{lem}

\begin{proof}
As in the proof of Lemma \ref{log},
we use a Vitali covering of the set $D_{k+1}$ with sections $S_h(x)$ of
normalized size $\alpha=C_0 M^k$, i.e.
$$D_{k+1}\, \subset \, \bigcup S_{h_i}(x_i), \quad \quad \mbox{$S_{\delta h_i}(x_i)$ disjoint sets.}$$ 
We then apply Lemma \ref{cors} above with $$K:=C_0^{-1}M,$$ so that $\alpha K=M^{k+1}$, and find that for each $i$
$$|D_{k+1}\cap S_{h_i}(x_i)| \le \frac{2C_0}{M\log(M)}  |D_k \cap S_{\delta h_i}(x_i)|,$$
provided that $M \gg C_0$.
Summing over $i$ and choosing $M \geq e^{4C_0}$ we get
$$ |D_{k+1}| \le \frac{2C_0}{M\log(M)}  |D_k|\leq  \frac{1}{2M} |D_k|,$$
and the lemma is proved by choosing $\eps=\log(2)/\log(M)$.
\end{proof}

\subsection{More general measures}\label{last sec}
It is not difficult to check that our proof applies to more general right hand sides.
Precisely we can replace $f$ by any measure $\mu$ of the form
\begin{equation}
\label{eq:mu}
\mu=\sum_{i=1}^N g_{i}(x) |P_i(x)|^{\alpha_i} \, dx, \quad \quad 0 < \lambda \le g_i \le \Lambda, \quad \text{$P_i$ polynomial,} \quad \alpha_i \geq 0.
\end{equation}
We state the precise estimate below.

\begin{thm}\label{thm2}
Let $u: \overline \Omega \to \R$,
\begin{equation*}
u=0 \quad \mbox{on $\p \Omega$}, \quad \quad B_1\subset \Omega \subset B_n,
\end{equation*}
be a continuous convex solution to the Monge-Amp\`ere equation
$$
\det D^2 u=\mu \quad \mbox{in $\Omega$}, \quad \quad \mu(\Omega) \le 1,
$$
with $\mu$ as in \eqref{eq:mu}. Then 
$$\|u\|_{W^{2,1+\eps}(\Omega')} \le C, \quad \quad \mbox {with} \quad \Omega':=\{ u< -\|u\|_{L^\infty}/2\} ,$$
where $\eps,C>0$ are universal constants.
\end{thm}

The proof follows as before, based on the fact that for $\mu$ as above one can prove the existence
of constants $\beta\geq 1$ and $\gamma>0$,
such that, for all convex sets $S$,\footnote{Although this will not be used here,
we point out for completeness that \eqref{eq:property} is equivalent to the so-called \emph{Condition $(\mu_\infty)$}, first introduced
by Caffarelli and Gutierrez in \cite{CG}. Indeed, using \eqref{eq:property} with $E=S\setminus F$
one sees that $|F|/|S| \ll 1$ implies $\mu(F)/\mu(S) \leq 1-\gamma/2$,
and then an iteration and
covering argument in the spirit of \cite[Theorem 6]{CG}
shows that \eqref{eq:property} is actually equivalent to \emph{Condition $(\mu_\infty)$}.}
\begin{equation}
\label{eq:property}
\frac{\mu(E)}{\mu(S)} \ge \gamma \left(\frac{|E|}{|S|}\right)^\beta\qquad \forall\,E\subset S,
\end{equation}

In this general situation, we need to write the scaling properties of $u$ with respect to the measure $\mu$.
More precisely the scaling inclusion \eqref{scal} becomes
$$\sigma \, h \,\mu(S_h(x_0))^ {- \frac 1n} \, \, B_1  \, \subset \, A(S_h(x_0)-x_0) \, \subset \, \sigma^{-1}   \, h \, \mu(S_h(x_0))^ {- \frac 1n} \, \, B_1,$$
and
$$Tx:=  h^{-1} \mu(S_h(x_0))^\frac 1n \, A \, (x-x_0). $$
Also we define the \textit{normalized size} $\alpha$ of $S_h(x_0)$ (relative to the measure $\mu$) as 
$$\alpha:=  h^{-1} \mu(S_h(x_0))^\frac 2n \, \|A\|^2.$$
With this notation the statements of the lemmas in Section 3 apply as before.

Indeed, first of all we observe that \eqref{eq:property} implies that $\mu$ is doubling,
so all properties of sections stated in Section \ref{eng}
still hold.

Then, in the proof of Lemma \ref{basic}, we simply apply
\eqref{eq:property} with $S=S_1(0)$ and $E=\{\det(D^2u) \leq c_2\}$ ($c_2>0$ small) to deduce that
$$
\gamma |E|^\beta \leq C \mu(E)=C\int_E \det(D^2u)\leq c_2 |E|.
$$
This implies that, if $c_2>0$ is sufficiently small,
the set
$$
\left\{\|D^2u\| \leq 2C_1c_1^{-1}\right\}\cap \left\{\det(D^2u) > c_2\right\}
$$
has at least measure
$c_1/4$, and the result follows as before.

Moreover, since \eqref{eq:property} is affinely invariant, Lemma \ref{basic_sc}
follows again from Lemma \ref{basic} by rescaling.
Finally, the proof of Lemma \ref{log} is identical.

\end{document}